\documentclass[12pt]{amsart}
\usepackage{geometry}                
\usepackage{graphicx}
\usepackage{amssymb}
\usepackage{epstopdf}
\newtheorem{lemma}{Lemma}[section]
\newtheorem{prop}[lemma]{Proposition}
\newtheorem{thm}[lemma]{Theorem}
\newtheorem{corollary}[lemma]{Corollary}
\newtheorem{ques}[lemma]{Question}

\DeclareMathOperator{\Vol}{Vol}
\DeclareMathOperator{\Dist}{Dist}
\DeclareMathOperator{\Diam}{Diam}
\DeclareMathOperator{\Star}{Star} 
\newcommand{\RR}{\mathbb{R}}

\title{Volumes of balls in Riemannian manifolds and Uryson width}
\author{Larry Guth}

\begin{document}

\begin{abstract} If $(M^n, g)$ is a closed Riemannian manifold where every unit ball has volume at most $\epsilon_n$ (a sufficiently small constant), then the $(n-1)$-dimensional Uryson width of $(M^n, g)$ is at most 1.

\end{abstract}

\maketitle

If $X$ is a compact metric space, then we say that $X$ has Uryson q-width $\le W$ if there is a
q-dimensional simplicial complex $Y$ and a continuous map $\pi: X \rightarrow Y$ so that every fiber
$\pi^{-1}(y)$ has diameter $\le W$ in $X$.  In other words, if $x_1, x_2 \in X$ lie in the same fiber,
$\pi^{-1}(y)$, then $d_X(x_1, x_2) \le W$.  We denote the Uryson q-width of $X$ by $UW_q(X)$.
If $X$ is homeomorphic to a q-dimensional simplicial complex, then we can choose $Y = X$ and $\pi$ to be the identity,
and so $UW_q(X) = 0$.  Roughly speaking $UW_q(X)$ is small if $X$ is ``close to being q-dimensional".  

To get a first perspective, let $S^1(L)$ denote a circle of length $L$ and let $X$ be the product
$S^1(W) \times S^1(L)$, where $W < L$.  The Uryson 1-width of this space is $\sim W$.  If we let $\pi$ denote the projection from $S^1(W) \times S^1(L)$ to $S^1(L)$, then each fiber of $\pi$ has diameter $W/2$.  Therefore, $UW_1(X) \le W/2$.  On the other hand, the Uryson 1-width of $X$ is $\gtrsim W$ because of the Lebesgue covering lemma, which we discuss more later in the introduction.  

The Uryson width originally appeared in the early 1900's in connection with topological dimension theory.  In \cite{Gr1}, Gromov began to investigate the Uryson widths of Riemannian manifolds in connection with systolic geometry.  In \cite{Gr3}, he made a conjecture about Uryson widths and the volumes of balls, which we prove in this paper.  

\begin{thm} \label{main} There exists $\epsilon_n > 0$ so that the following holds.  
If $(M^n, g)$ is a closed Riemannian manifold, and if there is a radius $R$ so that
every ball of radius $R$ in $(M^n, g)$ has volume at most $\epsilon_n R^n$, then $UW_{n-1}(M^n, g) \le R$.
\end{thm}

This theorem is a slightly stronger version of an estimate about filling radius from \cite{Gu2}.  The filling radius is another metric invariant introduced in \cite{Gr1}.   The filling radius of 
$(M^n, g)$ has a similar flavor to $UW_{n-1}(M^n, g)$.   For example, the filling radius of a product $S^1(W) \times S^1(L)$ is also around $W$ (assuming $W < L$).  It is not hard to show that $Fill Rad(M^n, g) \lesssim UW_{n-1}(M^n, g)$ -- see Appendix 1 of \cite{Gr1}, (B)-(D).  It is not known whether there is an inequality in the other direction.  I suspect that there is a closed manifold $M^n$ and a sequence of metrics $g_j$ where the ratio $UW_{n-1}(M^n, g_j) / Fill Rad(M^n, g_j)$ gets arbitrarily large, but I don't know of any examples.  If so, a bound on Uryson (n-1)-width is slightly stronger than a bound on filling radius.

One of the main theorems of \cite{Gr1} bounds the filling radius of a Riemannian manifold in terms of its volume.  

\begin{thm} \label{fillrad} (Gromov, \cite{Gr1}, Filling radius inequality) If $(M^n, g)$ is any closed Riemannian manifold, then $Fill Rad(M^n, g) \le C_n Vol (M^n, g)^{1/n}$.
\end{thm}

The paper \cite{Gr1} raised the question whether $UW_{n-1}(M^n, g)$ is also bounded in terms of the volume of $(M^n, g)$.  As a corollary of Theorem \ref{main}, we get such an estimate.

\begin{corollary} If $(M^n, g)$ is any closed Riemannian manifold, then $UW_{n-1}(M^n, g) \le C_n \Vol (M^n, g)^{1/n}$.
\end{corollary}

\begin{proof} We choose $R$ so that $\epsilon_n R^n = \Vol (M^n, g)$.  By Theorem \ref{main}, $UW_{n-1}(M^n, g) \le R = \epsilon_n^{-1/n} \Vol(M^n, g)^{1/n}$.  \end{proof}

In \cite{Gr3}, Gromov conjectured Theorem \ref{main}, and he also conjectured a slightly weaker theorem about the filling radius.  The paper \cite{Gu2} proved the filling radius version of this inequality.  

\begin{thm} \label{locfillrad} (Local filling radius inequality) If $(M^n, g)$ is a closed Riemannian manifold, $R > 0$, and every ball in
$(M^n, g)$ of radius $R$ has volume $< \epsilon_n R^n$, then the filling radius of $(M^n, g)$ is $\le R$.
\end{thm}

Theorem \ref{main} is a slightly stronger version of the local filling radius inequality, Theorem \ref{locfillrad}, which in turn is a stronger version of the filling radius inequality, Theorem \ref{fillrad}.  For more context about metric geometry and systolic geometry, the reader can consult \cite{Gr1} or \cite{Gu3}.  

The proof of Theorem \ref{main} closely follows the proof of Theorem \ref{locfillrad} with one new ingredient, which we call a pushout lemma for small surfaces.  This pushout lemma is the new observation in this paper.  

\subsection{A push-out lemma for small surfaces}

In the late 1950's, Federer-Fleming proved (a close relative of) the following result.

\begin{lemma} \label{ffpushout} (Federer-Fleming push-out lemma) Suppose that $X$ is a compact piecewise-smooth $n$-dimensional manifold with boundary.  Suppose that $K \subset \mathbb{R}^N$ is a bounded, open convex set, and $\phi_0: (X, \partial X) \rightarrow (K, \partial K)$ is a piecewise smooth map.  If $n < N$,
then $\phi_0$ may be homotoped to a map $\phi_1: X \rightarrow \partial K$ so that the following holds.

1. The map $\phi_1$ agrees with $\phi_0$ on $\partial X$.

2. $\Vol_n \phi_1(X) \le C(K,N,n) \Vol_n \phi_0(X)$.

\end{lemma}

This Lemma plays an important role in metric geometry.  See \cite{Gu1} for a proof and for a description of the connection with the filling radius inequality.  In the proof of Theorem \ref{main}, we need to construct a homotopy of this type, but we need a much more careful bound for the volume of $\phi_1(X)$.  We need to find a homotopy where $\Vol_n \phi_1 (X)$ is only very slightly larger than $\Vol_n \phi_0(X)$.  We show that if $\Vol_n \phi_0(X)$ is small, then we can homotope $\phi_0$ to $\phi_1$ so that $\phi_1$ lies near to $\partial K$ and $\Vol_n \phi_1 (X) \le \Vol_n \phi_0(X)$.  
Here is the statement of our pushout lemma for small surfaces.  

\begin{lemma} \label{pushoutlemma} (Pushout lemma for small surfaces) For each dimension $n \ge 2$, there is a constant $\sigma_n$ so that the following holds.  

Suppose that $X$ is a compact piecewise-smooth $n$-dimensional manifold with boundary.  
Suppose that $K \subset \mathbb{R}^N$ is a (bounded open) convex set, and $\phi_0: (X, \partial X) \rightarrow (K, \partial K)$ is a piecewise-smooth map.    Let $W = \sigma_n \Vol_n \phi_0(X)^{1/n}$.
Then $\phi_0$ may be homotoped to a map $\phi_1: (X, \partial X) \rightarrow (K, \partial K)$ so that the following holds.

1. The map $\phi_1$ agrees with $\phi_0$ on $\partial X$.

2. $\Vol_n \phi_1(X) \le \Vol_n \phi_0(X)$.

3. The image $\phi_1(X)$ lies in the $W$-neighborhood of $\partial K$.

\end{lemma}

When $K$ is a ball, we have the following corollary.

\begin{corollary} For every $n$ and every $\epsilon > 0$, we can find $\delta > 0$ so that the following holds.  Suppose that $X$ is a compact piecewise-smooth $n$-dimensional manifold with boundary.  Let $\phi_0: (X, \partial X) \rightarrow (B^N, \partial B^N)$ be a map to the unit N-ball.  If $\Vol_n \phi_0(X) < \delta$, then we can homotope $\phi_0$ to $\phi_2: X \rightarrow \partial B^N$
so that $\phi_2$ agrees with $\phi_0$ on $\partial X$ and $\Vol_n \phi_2(X) < (1 + \epsilon) \Vol_n \phi_0(X)$.

\end{corollary}

\begin{proof} By Lemma \ref{pushoutlemma}, we can homotope $\phi_0$ rel $\partial X$ to a map $\phi_1$ with $\Vol_n \phi_1(X)
\le \Vol_n \phi_0(X)$, and with $\phi_1(X)$ lying within the $W$-neighborhood of $\partial B^N$ for $W = \sigma_n \Vol_n \phi_0(X)^{1/n}
\le \sigma_n \delta^{1/n}$.  Now we radially project $\phi_1$ into $\partial B^N$, pushing away from the origin.  The resulting map is
$\phi_2$.  Since $\phi_1$ already
lies in the $W$-neighborhood of $\partial B^N$, the Lipschitz constant of this radial projection on $\phi_1(X)$ is $\le (1 - W)^{-1}$.  Hence
$\Vol_n \phi_2(X) \le (1- W)^{-n} \Vol_n \phi_0(X)$.  If we choose $\delta$ sufficiently small, we can arrange that  $(1 - W)^{-n} \le (1 + \epsilon)$. \end{proof}

We remark that without the assumption that $\Vol_n \phi_0(X)$ is small, we cannot hope to push $\phi_0$ into $\partial B^N$ without
stretching it significantly.  For example, if $\phi_0(X)$ is a flat n-disk through the origin of $B^N$, and if $\phi_0(\partial X)$ is the boundary of
the disk, then any homotopic map $\phi_2$ will have to have volume at least equal to that of an $n$-dimensional unit hemisphere.  The corollary
says that for very small surfaces, we can push the surface into $\partial B^N$ with only slight stretching.  

The intuition for our push-out lemma comes from minimal surface theory.  Suppose we consider all of the maps $\phi$ homotopic
to $\phi_0$ rel $\partial X$.  We let $V$ denote the infimal volume of $\phi (X)$ among all these maps.  Clearly $V \le \Vol_n \phi_0(X)$.  
It seems plausible that $V$ is actually realized by a map $\phi_1$ whose image is some kind of stationary object, such as a
stationary varifold, in the interior of $K$.  For the sake of intuition let us suppose that such a $\phi_1$ exists.  Since the image is a stationary varifold in the interior of $K$, it obeys the monotonicity formula.  Namely, for any regular point $p$ in $\phi_1(X)$ and
any $B(p,r) \subset K$, we have $\Vol_n \phi_1(X) \cap B(p,r) \ge \omega_n r^n$, where $\omega_n$ is the volume of the unit $n$-ball.
We see that $\omega_n r^n \le \Vol_n \phi_0(X)$.  Hence $\phi_1(X)$ lies in the $W$-neighborhood of $\partial K$, for 
$W = \omega_n^{-1/n} \Vol_n \phi_0(X)^{1/n}$.  This thought experiment provides intuition for our push-out lemma, and it also suggests what the sharp constant $\sigma_n$ should be: $\omega_n^{-1/n}$.

The proof of Theorem \ref{locfillrad} used some minimal surface theory and the monotonicity formula.  We give a direct constructive proof of Lemma \ref{pushoutlemma}, so no minimal surface theory is required in the proof of Theorem \ref{main}.  Now Theorem \ref{main} implies Theorem \ref{locfillrad}, so in particular, we get a modified proof of the local filling radius inequality without any use of minimal surface theory.  The proof of Lemma \ref{pushoutlemma} is based on Wenger's recent proof of Gromov's filling volume inequality, \cite{W}.  

\subsection{Background on Uryson width in topological dimension theory}

The Uryson width first appeared in topological dimension theory in the early 1900's.  In this subsection, we review some topological dimension theory.  We will see the role of Uryson width in topology, and we will compare our main theorem to some classical results of topological dimension theory.   In particular, we will see that our Theorem \ref{main} is a more quantitative version of a classical theorem of Szpilrajn comparing topological dimension with Hausdorff dimension.  We will outline the proof of Szpilrajn's theorem, and we will see how ``push out'' homotopies, as in Lemma \ref{pushoutlemma}, come into the story.  A number of other characters in the proof of Theorem \ref{main} also come from topological dimension theory.  

We first note that the Uryson width has an equivalent definition in terms of
open covers.  If $\{ O_i \}$ is an open cover of $X$, we say that it has multiplicity
$\le m$ if each point of $x$ lies in $\le m $ different sets $O_i$.  We say that
the diameter of the open cover is $\sup_i \Diam(O_i)$.

Open covers and Uryson width are connected via the idea of the nerve of a covering.  We recall the definition of the nerve of a cover to set the notation.  Suppose that $\{ O_i \}_{i=1, ..., S}$ is an open cover of $X$.  The nerve $N$ of this cover is a simplicial complex with one vertex $v_i$ for each set $O_i$.  Given a $(k+1)$-tuple of vertices, $v_{i_0}, ..., v_{i_k}$, there is a simplex in the nerve with these vertices if and only if $\cap_{j=0}^k O_{i_j}$ is non-empty.  The nerve $N$ naturally sits in $[0,1]^S$, where the vertex $v_i$ corresponds to the unit vector in the $i^{th}$ direction.  In particular, any map $F: X \rightarrow N$ has coordinates $F_1, ..., F_S$.  We say that the map $F$ is subordinate to the cover $\{ O_i \}$ if each $F_i$ is supported in $O_i$.  

\begin{lemma} \label{uryleb} A compact metric space $X$ has $UW_q(X) < W$
if and only if there is an open cover of $X$ with multiplicity $\le q+1$
and diameter $< W$.
\end{lemma}

We sketch the proof of Lemma \ref{uryleb}.  Suppose that $X$ has an open cover with multiplicity $\le q+1$
and diameter $< W$.  Let $N$ denote the nerve of the cover.  Note that
$N$ is a simplicial complex of dimension $q$.  Let $\phi_i$ be a partition of unity
on $X$ with $\phi_i$ supported in $O_i$.  These $\phi_i$ give a continuous map $\phi$
from $X$ to the nerve $N$.  Note that $\phi$ is subordinate to the cover $\{ O_i \}$.  In particular, 
each fiber $\phi^{-1}(y)$ is contained in one of the sets $O_i$, and so it has diameter $< W$.  Hence $UW_q(X) < W$.

Now suppose that $UW_q(X) < W$.  Let $\pi: X \rightarrow Y$ be a map to 
a q-dimensional polyhedron so that each fiber has diameter $\le UW_q(X) + \epsilon < W$.
Since $X$ is compact, the image of $X$ is compact, and so we can assume that
$Y$ is compact.  A q-dimensional polyhedron has an open cover $O_i'$ with multiplicity at most $q+1$ where
the diameters of the sets $O_i'$ can be made arbitrarily small.  Then, $O_i = \pi^{-1}(O_i')$ give an open cover of $X$ with multiplicity at most $q+1$.  If the diameter of $O_i'$ is small enough, then each $O_i$ lies in an arbitrarily small neighborhood of a fiber $\pi^{-1}(y)$, and so each $O_i$ has diameter $< W$.  

The most famous estimate about Uryson width is the Lebesgue covering lemma.

\newtheorem*{lcl}{Lebesgue covering lemma}

\begin{lcl} Suppose that $\{ O_i \}$ is an open cover of the unit n-cube
$Q_n = [0,1]^n$ with multiplicity $\le n$.  Then the diameter of $\{ O_i \}$ is
$\ge 1$.  In other words, $UW_{n-1}(Q_n) = 1$.
\end{lcl}

The Lebesgue covering lemma was stated by Lebesgue and first proven by
Brouwer a little later.  For a proof, see \cite{HW}.   The Lebesgue covering lemma has many applications in topological dimension theory.  Since $UW_{n-1}(Q_n) > 0$, there is no injective continuous
map from $Q_n$ into an (n-1)-dimensional polyhedron.  In particular, there is no injective
continuous map from $\mathbb{R}^n$ into $\mathbb{R}^{n-1}$.  This implies the famous
topological invariance of dimension.

\newtheorem*{tid}{Topological invariance of dimension}

\begin{tid} (Brouwer, 1909) If $m \not= n$, then $\mathbb{R}^m$ and $\mathbb{R}^n$
are not homeomorphic.
\end{tid}

The Lebesgue covering dimension of a compact metric space $X$ is the smallest
$q$ so that $UW_q(X) = 0$.  In particular, the Lebesgue covering dimension of
$Q_n$ is equal to $n$, which follows from the Lebesgue covering lemma.

If $X_1$ and $X_2$ are homeomorphic compact metric spaces, then the Lebesgue covering dimension
of $X_1$ is equal to that of $X_2$.  In other words, $UW_q(X_1) = 0$ if and only if $UW_q(X_2) = 0$.
This follows fairly easily from the definition of Uryson width and the following observation.  
If $F: X_1 \rightarrow X_2$ is a homeomorphism
of compact metric spaces, then for every $\epsilon$ there exists a $\delta$ so that if $A \subset X_2$ has
diameter $< \delta$ then $F^{-1}(A) \subset X_1$ has diameter $< \epsilon$.  

(The Lebesgue covering lemma is a precise quantitative estimate for $UW_{n-1}(Q_n)$.  It easily implies non-sharp lower bounds on the Uryson (n-1)-width of other manifolds.  For instance, it follows that the Uryson (n-1)-width of the unit n-sphere is $\gtrsim 1$ with an explicit lower bound.  This was sharpened by Katz in \cite{K}, who proved the precise value of $UW_{n-1}(S^n)$.)

There are several notions of the ``dimension" of a compact metric space $X$, and it is interesting
to see how they relate.  One important notion is the Hausdorff dimension.  Szpilrajn proved that the
Lebesgue covering dimension of any $X$ is at most the Hausdorff dimension of $X$.  In fact, he
proved an even stronger theorem, which we now state.

\newtheorem*{sdi}{Szpilrajn dimension inequality}

\begin{sdi} Let $X$ be a compact metric space with n-dimensional Hausdorff measure equal to zero.
Then $UW_{n-1}(X) = 0$.  In other words, the Lebesgue covering dimension of $X$ is $\le n-1$.
\end{sdi}

(Szpilrajn's inequality was pointed out to me by Anton Petrunin and Vitali Kapovitch.)  

 Our main theorem is a quantitative version of the Szpilrajn theorem.  In particular, Szpilrajn's theorem implies that if $X$ is a compact metric space and every unit ball in $X$ has
n-dimensional Hausdorff measure exactly equal to zero, then $UW_{n-1}(X) \le 1$.  In Theorem \ref{main}, we suppose that $X$ is a closed $n$-dimensional Riemannian manifold.  Instead of assuming that every unit ball in $X$ has zero n-dimensional volume, we assume that every unit ball has n-dimensional volume at most a tiny constant $\epsilon_n$.  Under this weaker assumption, Theorem \ref{main} says that $UW_{n-1}(X)$ is still at most 1, provided that $\epsilon_n$ is small enough.  

We sketch here the proof of Szpilrajn's theorem.  The proof of Theorem \ref{main} follows the same outline as this proof, but it requires careful estimates at several steps.  

\begin{proof} Pick any $\epsilon > 0$.  Cover $X$ by finitely many open sets $O_i$ with diameter
$< \epsilon$.  Let $\phi: X \rightarrow N$ be a map to the nerve coming from a partition of unity on $X$, as in the discussion of Lemma \ref{uryleb} above.  

The idea of a mapping subordinate to the cover $\{ O_i \}$ is central to the proof.  Recall that a map $F: X \rightarrow N$ is subordinate to the cover $\{ O_i \}$ if each component $F_i$ is supported in $O_i$.  The map $\phi$ coming from a partition of unity (subordinate to the cover $\{ O_i \}$) is indeed subordinate to the cover.  
We can also arrange that the map $\phi$ is Lipschitz.  
Therefore, the image $\phi(X)$ has n-dimensional Hausdorff measure zero.

Let $D$ denote the dimension of the nerve $N$.  We will construct a sequence of homotopic maps from $X$ to $N$, 
$\phi = \phi^{(D)} \sim \phi^{(D-1)} \sim ... \sim \phi^{(n)} \sim \phi^{(n-1)}$, where
$\phi^{(k)}$ maps $X$ into the k-skeleton of the nerve $N$.  We denote the k-skeleton of $N$ by $N^{(k)}$.
All maps $\phi^{(k)}$ are Lipschitz and all maps subordinate to the cover.

For $n \le k \le D$, we homotope $\phi^{(k)}$ to $\phi^{(k-1)}$ as follows.  Since $\phi^{(k)}$ is Lipschitz,
$\phi^{(k)}(X)$ has n-dimensional Hausdorff measure zero.  Since $k \ge n$, the image 
$\phi^{(k)}(X)$ does not contain any k-face of $N$.  For each k-face $F_j \subset N$, we pick a
point $y_j \in F_j$ with $y_j$ not in the image of $\phi^{(k)}$.  We let $\tilde N^{(k)}$ denote 
the k-skeleton of $N$ with the points $y_j$ removed.

We let $R_k$ be a retraction from $\tilde N^{(k)}$ into $N^{(k-1)}$, taking $F_j \setminus \{y_j\}$ into
$\partial F_j$.  We define $\phi^{(k-1)}$ to be $R_k \circ \phi^{(k)}$.  We can arrange that 
the map $R_k$ is Lipschitz on any compact subset of $\tilde N^{(k)}$.  Therefore, $\phi^{(k-1)}$ is
Lipschitz.

Since $R_k$ maps each face $F \subset N^{(k)}$ into its closure $\bar F$, and since $\phi^{(k)}$ is
subordinate to the cover, it follows that $\phi^{(k-1)}$ is subordinate to the cover as well.

In particular, we get a map $\phi^{(n-1)}$ from $X$ to the (n-1)-dimensional polyhedron $N^{(n-1)}$ 
subordinate to the cover.  Because $\phi^{(n-1)}$ is subordinate to the cover, each fiber of the map
lies in one of the sets $O_i$ and has diameter $< \epsilon$.  Hence $UW_{n-1}(X) < \epsilon$.

\end{proof}

The proof of Theorem \ref{main} follows this outline, but it requires careful estimates.  By assumption, each unit ball of $X$ has small volume, and so each set $O_i$ in the cover can be taken to have small volume.  The map $\phi = \phi^{(D)}$ is qualitatively Lipschitz, but we need to control the Lipschitz constant of $\phi$ in order to bound the volume of $\phi(O_i)$.  This part of the argument is handled by ideas from \cite{Gu2}.  Suppose now that $\phi(O_i) = \phi^{(D)}(O_i)$ has very small volume for each $i$.  We will construct a sequence of homotopies, $\phi^{(D)} \sim \phi^{(D-1)} \sim ... \sim \phi^{(n)}$.  At each step, we need to prove that $\phi^{(k)}(O_i)$ remains small.  If $\Vol_n \phi^{(k)}(O_i) = 2 \Vol_n \phi^{(k+1)}(O_i)$, then it turns out that $\phi^{(n)}(O_i)$ will be far too large.  So we need to homotope $\phi^{(k+1)}$ to $\phi^{(k)}$ while being very careful about how much volumes stretch.  This step is accomplished with Lemma \ref{pushoutlemma}, the pushout lemma for small surfaces.  

For more information about topological dimension
theory, one can read the book \cite{HW}.  For a more geometric point of view
about Uryson width see the survey paper \cite{Gr2}.

To end this discussion, we remark that trying to transform a qualititative theorem from topological dimension theory into a quantitative estimate in Riemannian geometry doesn't always work.  For example we consider the following theorem on the locality of topological dimension.

\newtheorem*{localdim}{Locality of Lebesgue covering dimension}
\begin{localdim} Let $X$ be a compact metric space.  Suppose that $X$ is covered by open sets
$O_i$ and suppose that $UW_q(O_i) = 0$ for each $i$.  Then $UW_q(X) = 0$.  \end{localdim}

This locality theorem follow from the
central results of topological dimension theory in \cite{HW}.  The main theme of the book \cite{HW} is
that various ways of defining the `topological dimension' of a compact metric space are equivalent to
each other.  The central definition in the book is the topological dimension (or Menger topological dimension)
defined inductively as follows.  A compact metric space has topological dimension -1 if it is empty.
A compact metric space $X$ has topological dimension $\le q$ at a point $x \in X$ if $x$ has arbitrarily small
open neighborhoods $U$ such that $\partial U$ has topological dimension $\le q-1$.  A compact metric
space $X$ has topological dimension $\le q$ if it has topological dimension $\le q$ at each point $x \in X$.
Theorem V8 of \cite{HW} says that for a compact metric space $X$, the topological dimension and the Lebesgue covering dimension are the same.  But the definition of topological dimension is clearly local, because it only involves
arbitrarily small neighborhoods of every point.  If $X$ is covered by open sets $O_i$, and each set $O_i$ has topological dimension at most $q$, then it follows immediately that $X$ has topological dimension at most $q$.  Since the Lebesgue covering dimension is equivalent to the topological dimension, the Lebesgue covering dimension is also local.  Finally, $X$ has Lebesgue covering dimension at most $q$ if and only if $UW_q(X) = 0$.  

We consider a more quantitative version of this result for Riemannian manifolds:

\begin{ques} \label{locury} Is there a constant $\epsilon(q,n) > 0$ so that the following holds: if $(M^n, g)$ is a closed Riemannian manifold, and every unit ball in $(M, g)$ has Uryson $q$-width at most $\epsilon(q,n)$, then $(M^n, g)$ has Uryson $q$-width at most 1? 
\end{ques}

This question is related to the locality of dimension theorem in exactly the same way that Theorem \ref{main} is related to Szpilrajn's theorem.  But the answers turn out to be different.  At least when $n=3$ and $q=2$, the answer to question \ref{locury} is no because of the following counterexample:

\begin{prop} \label{counterex} For any $\epsilon > 0$, there is a metric $g_\epsilon$ on $S^3$ so that
every unit ball in $(S^3, g_\epsilon)$ has Uryson 2-width $< \epsilon$ and yet the whole manifold $(S^3, g_\epsilon)$
has Uryson 2-width at least 1.
\end{prop}

The main idea of this counterexample comes from \cite{Gr2}, Example $(H_1'')$.

So from a quantitative point of view, the Szpilrajn theorem is more robust than the locality theorem.

\subsection{Outline of the paper}

In Section 1 of the paper, we prove Theorem \ref{main} using the pushout lemma for small surfaces.  The proof closely follows the arguments from \cite{Gu2}, and we use some lemmas from \cite{Gu2}.  In Sections 2 and 3, we prove the pushout lemma for small surfaces.  Section 2 introduces a key tool, a variation on the isoperimetric inequality which we call an isoperimetric extension lemma.  Section 3 uses the extension lemma to prove Lemma \ref{pushoutlemma}.  In Section 4, we prove Proposition \ref{counterex}.  In Section 5 we discuss a few open questions, and in Section 6, we discuss complete manifolds and manifolds with boundary.

\section{The proof of the main theorem}

The proof of the main theorem is closely based on the argument in \cite{Gu2}.  The new ingredient is the pushout lemma for small surfaces, which we will prove in Section 3 below.  

Suppose $(X^n, g)$ is a closed Riemannian manifold of dimension $n$.  We 
assume that each unit ball in $(X^n, g)$ has volume $< \epsilon(n)$, a sufficiently small number that
we can choose later.  We have to construct a continuous map $\pi: X \rightarrow Y^{n-1}$ to
a polyhedron of dimension $n-1$ so that each fiber $\pi^{-1}(y)$ has diameter at most 1.

Suppose that $B_i$ in $X$, $i=1, ..., D$, so that $\frac{1}{2} B_i$ cover $X$ and so that each ball $B_i$ has radius $< (1/100)$.  We will choose a particular set of balls $B_i$ later, obeying some additional geometric estimates.  Next we build a map to the ``rectangular nerve" of this cover.

The rectangular nerve of the covering $\{ B_i \}$ is defined as follows.  We begin with the rectangle $\prod_{i=1}^D [0, r_i]$
where $r_i$ is the radius of the ball $B_i$.  We let $\phi_i$, $i = 1, ..., D$ be the coordinate functions on this
rectangle.  The rectangular nerve is a closed sub-complex of this product.
An open face $F$ of the rectangle is determined by dividing the dimensions $1, ..., D$ into three
sets: $I_0, I_1,$ and $I_{(0,1)}$.  The open face $F$ is defined by the equations $\phi_i = 0$ if
$i \in I_0$, $\phi_i = r_i$ if $i \in I_1$, and $0 < \phi_i < r_i$ if $i \in I_{(0,1)}$. We let $I_+$ be the
union of $I_1$ and $I_{(0,1)}$.  Now an open face $F$ is contained in the rectangular nerve $N$
iff the set $I_1(F)$ is not empty, and the intersection $\cap_{i \in I_+(F)} B_i$ is not empty.

As with the regular nerve, we can construct a naturally defined map from $X$ to the rectanglar nerve.  We call
the map $\phi: X \rightarrow N$.   To define $\phi$, we let $\phi_i$ be a real-valued function supported 
on $B_i$ with $\phi_i(x) = r_i$ for $x \in \frac{1}{2} B_i$ and $\phi_i(x)$ decreasing 
to zero as $x$ approaches $\partial B_i$.  We can choose $\phi_i$ to be piecewise smooth with values in 
$[0, r_i]$, and with Lipschitz constant less than 3.  Taking $\phi_i$ as coordinates, we get a map $\phi: X \rightarrow \prod_{i=1}^D [0, r_i]$,
and the image of $\phi$ lies in the rectangular nerve $N$.  To see this, suppose that $\phi(x)$ lies in an open
face $F$.  Because $\frac{1}{2} B_i$ cover $M$, $\phi_i(x) = r_i$ for some $i$, and so $I_1(F)$ is not empty.
Since $\phi_i$ is supported in $B_i$, $x$ lies in $\cap_{i \in I_+(F)} B_i$, which must be non-empty.

The polyhedron $Y$ will be the (n-1)-skeleton of the rectangular nerve $N$.  To construct $\pi: X \rightarrow Y$,
we will homotope $\phi$ until its image lands in $Y$.

We say that a map $\psi: X \rightarrow N$ is subordinate to our covering if the $i^{th}$ component, $\psi_i$, is supported
in $B_i$ for each $i = 1, ..., D$.   (Since $N \subset \prod_{i=1}^D [0, r_i]$, we can speak of the $i^{th}$ coordinate of $\psi$.)  By construction our map $\phi$ is subordinate to the cover.  The map $\pi: X
\rightarrow Y \subset N$ will also be subordinate to the covering.  If $\psi: X \rightarrow N$ is subordinate to the
covering, then any inverse image $\psi^{-1}(y)$ is contained in some ball $B_i$ of radius $< 1/ 100$.  So any
fiber of $\psi$ has diameter $\le 1/50 < 1$.

The strategy for constructing $\pi$ is based on the proof of Szpilrajn's theorem.  We let $\phi^{(D)}$ denote
our original map $\phi$, and we construct a sequence of homotopies $\phi^{(D)} \sim \phi^{(D-1)} \sim
... \sim \phi^{(n)} \sim \phi^{(n-1)}$.  Each map $\phi^{(k)}$ maps $X$ to the k-skeleton of $N$.  And each of
these maps is subordinate to the covering.  The map $\phi^{(n-1)}$ is our desired map $\pi$, and to prove
our theorem, we only need to construct a map $\phi^{(n-1)}$ from $X$ to the $(n-1)$-skeleton of $N$ which is subordinate to the cover.

Since we will need to check that various maps are subordinate to the cover, the following simple observation is useful.

\begin{lemma} \label{subcheck} Suppose that $\Phi_1: X \rightarrow N$ is subordinate to the cover, and that $\Phi_2: X
\rightarrow N$ is another map.  Suppose that for each $x \in X$, if $\Phi_1(x)$ lies in an open face $F \subset
N$, then $\Phi_2(x)$ lies in its closure $\bar F$.  Then $\Phi_2$ is also subordinate to the cover.
\end{lemma}

\begin{proof} Pick an index $i$.  We want to show that $i^{th}$ coordinate of $\Phi_2$ is supported in 
$B_i$.  Pick a point $x \in X$ with $\Phi_{2,i} (x) \not= 0$.  The image $\Phi_2(x)$ lies in an open
face $F$ with $i \in I_+(F)$.  Let us say that $\Phi_1(x)$ lies in the open face $F_1$.  By our hypothesis
$F \subset \bar F_1$.  Hence $I_+(F) \subset I_+(F_1)$.  In particular, $i \in I_+(F_1)$, and so 
$\Phi_{1,i}(x) > 0$.  Since $\Phi_1$ is subordinate to the cover, $x \in B_i$.  \end{proof}

This lemma applies to the push-out type construction that we used in Section 1 to prove Szpilrajn's
theorem.  In particular, we get the following lemma.

\begin{lemma} \label{easypush} Suppose that $\phi^{(k)}: X \rightarrow N^{(k)}$ is a map from $X$ to the k-skeleton of
$N$ subordinate to the cover.  Suppose that for each k-face $F \subset N$, the image $\phi^{(k)}(X)
\cap F$ is not the entire k-face.  Then we can homotope $\phi^{(k)}$ to a map $\phi^{(k-1)}: X
\rightarrow N^{(k-1)}$ subordinate to the cover.
\end{lemma}

\begin{proof}  For each k-face, $F_j \subset N$, pick a point $y_j$ which lies in $F_j$ but does not lie
in $\phi^{(k)}(X)$.  We let $N'$ denote the k-skeleton of $N$ take away the points $y_j$ we just chose.
So $\phi^{(k)}$ maps $X$ into $N'$.  Now we let $\Psi$ be a retraction from $N'$ to $N^{(k-1)}$, constructed
by radially pushing $F_j \setminus \{ y_j \}$ into $\partial F_j$.  We define $\phi^{(k-1)}$ to be
$\Psi \circ \phi^{(k)}: X \rightarrow N^{(k-1)}$.  Since $\Psi$ maps each face $F_j$ into $\bar F_j$, Lemma \ref{subcheck} implies that
$\phi^{(k-1)}$ is subordinate to the cover. \end{proof}

Using this lemma, we can construct a sequence of homotopies subordinate to the cover, $\phi^{(D)}
\sim ... \sim \phi^{(n)}$.  As long as $k > n$, any piecewise smooth map $\phi^{(k)}: X
\rightarrow N^{(k)}$ will not contain any k-face in its image, and so we can apply Lemma \ref{easypush}.
The real obstruction comes when we try to homotope $\phi^{(n)}$ to $\phi^{(n-1)}$.

In the case of the Spzilrajn theorem, we knew that $\phi^{(n)}(X)$ has n-dimensional Hausdorff measure 0.
Hence $\phi^{(n)}(X)$ does not cover any n-dimensional face $F \subset N$, and we can Lemma \ref{easypush}.  In our case, we will prove the following volume estimate.

$$ \textrm{For each n-face} \hskip6pt F \subset N, \Vol_n \phi^{(n)}(X) \cap F < \Vol_n F. \eqno{(*)}$$

(The rectangular nerve has a metric given by the restriction the Euclidean metric on the rectangle $\prod_{i=1}^D [0, r_i]$.)

Given this key estimate $(*)$, we can apply Lemma \ref{easypush} again to homotope $\phi^{(n)}$ to $\phi^{(n-1)}$ subordinate to the cover.  This will prove the main theorem.

Establishing the volume estimate $(*)$ takes some care.  Since $\phi^{(n)}$ is subordinate to our
covering, the preimage of any face $F$ lies one of our balls $B_i$.  The ball $B_i$ has volume at most $\epsilon$,
which we can choose very small.  But the initial map $\phi = \phi^{(D)}$ may stretch this volume.
Then each homotopy from $\phi^{(k)}$ to $\phi^{(k-1)}$ may stretch it further.  After all this potential
stretching, we need a bound for the volume of $\phi^{(n)}(X) \cap F$.  To get such a bound, we need to choose
the covering carefully, in order to control the volume of $\phi^{(D)}(X) \cap F$.  Then we need to choose our homotopies
from $\phi^{(k)}$ to $\phi^{(k-1)}$ carefully, in order to inductively control the volume of each $\phi^{(k-1)}(X) \cap F$.

We need a little bit of vocabulary in order to state our estimates.  Given an open face $F \subset N$, we let
$\Star(F)$ denote the union of $F$ and all open faces $F'$ such that $F \subset \bar F'$.  If $F$ has dimension
$k$, then each open face in $F' \subset \Star(F)$ has dimension $\ge k$, and the only k-face in $\Star(F)$
is $F$.  We let $r_1(F)$ denote the shortest length of any of the sides of $F$.

The paper \cite{Gu2} constructs a good covering $\{ B_i \}$ and proves the estimates that we need about $\phi = \phi^{(D)}$.  In particular, we use Lemma 5 from \cite{Gu2}:

\begin{lemma} \label{Gu2lemma} (Lemma 5 in \cite{Gu2}.)  There are constants $C(n)$ and $\beta(n) > 0$ depending only on $n$ so that the following
estimate holds.  Let $\epsilon > 0$ be any number.  Suppose that $(X^n, g)$ is a closed Riemannian manifold of dimension $n$, and that each unit ball in $(X^n, g)$ has volume $< \epsilon$.  Then there is a covering $B_i$ as above and a map $\phi^{(D)}: X \rightarrow N$ subordinate to the cover so that that following volume estimate holds:

For any face $F \subset N$ of dimension $d(F)$,

\begin{equation} \label{volestD} \Vol_n [ \phi^{(D)}(X) \cap \Star(F) ] \le C \epsilon r_1(F)^{n+1} e^{- \beta d(F)} . \end{equation}

\end{lemma}

We will construct a sequence of homotopic maps $\phi = \phi^{(D)} \sim
\phi^{(D-1)} \sim ... \sim \phi^{(n)}$ subordinate to the cover.  Moreover, every map $\phi^{(k)}$ will obey the following estimate, slightly weaker than the estimate
that $\phi$ obeys.

$$\Vol_n [\phi^{(k)}(X) \cap \Star(F)] < 2 C \epsilon r_1(F)^n e^{- \beta d(F)}. \eqno{(**)}$$

In particular, for each n-face $F$, $\phi^{(n)}(X) \cap F$ obeys the
following estimate:

$$\Vol_n [\phi^{(n)}(X) \cap F] < 2 C \epsilon r_1(F)^n.$$

We will choose $\epsilon$ less than $(2C)^{-1}$, so we conclude
that $\Vol_n [\phi^{(n)}(X) \cap F] < \Vol_n (F)$, proving $(*)$.  Then by Lemma \ref{easypush}, $\phi^{(n)}$ is homotopic to a
map $\phi^{(n-1)}$ to the $(n-1)$-skeleton of $N$, subordinate to the cover, and this
proves the main theorem.  So it suffices to construct
the maps $\phi^{(k)}$ for each $k = D, D-1, ..., n$, subordinate to the cover and obeying the key estimate $(**)$.

The homotopy from $\phi^{(k)}$ to $\phi^{(k-1)}$ is based on our pushout lemma, Lemma \ref{pushoutlemma}. We recall the statement here.

\newtheorem*{pushout}{Push-out lemma for small surfaces}

\begin{pushout} For each dimension $n \ge 2$, there is a constant $\sigma_n$ so that the following holds.  

Suppose that $X$ is a compact piecewise-smooth $n$-dimensional manifold with boundary.  
Suppose that $K \subset \mathbb{R}^N$ is a convex set, and $\phi: (X, \partial X) \rightarrow (K, \partial K)$ is a piecewise-smooth map.    Let $W = \sigma_n \Vol_n \phi_0(X)^{1/n}$.
Then $\phi$ may be homotoped to a map $\tilde \phi: (X, \partial X) \rightarrow (K, \partial K)$ so that the following holds.

1. The map $\tilde \phi$ agrees with $\phi$ on $\partial X$.

2. $\Vol_n \tilde \phi(X) \le \Vol_n \phi(X)$.

3. The image $\tilde \phi(X)$ lies in the $W$-neighborhood of $\partial K$.

\end{pushout}

We will apply this lemma on each k-face $F_j \subset N^{(k)}$.  We let $K_j \subset F_j$ be a closed convex
subset of $F_j$, containing almost all of $F_j$ but in general position.  We let $X_j$ be the preimage of $K_j$
under $\phi^{(k)}$.  So we have a map $\phi^{(k)}: (X_j,
\partial X_j) \rightarrow (K_j, \partial K_j)$, and we can apply the pushout lemma for small surfaces to this map.
The lemma gives us a new map $\tilde \phi^{(k)}: (X_j, \partial X_j) \rightarrow (K_j, \partial K_j)$ agreeing with
$\phi^{(k)}$ on $\partial X_j$.

We have just defined $\tilde \phi^{(k)}$ on $\cup_j X_j \subset X$.  We extend $\tilde \phi^{(k)}$ to all of $X$ by letting
$\tilde \phi^{(k)} = \phi^{(k)}$ on $X \setminus \cup_j X_j$.  The boundary between $X \setminus \cup_j X_j$ and $\cup X_j$
is $\cup \partial X_j$.  Since $\tilde \phi^{(k)} $ agrees with $\phi^{(k)}$ on each $\partial X_j$, this definition gives a piecewise smooth map
from $X$ to $N^{(k)} \subset N$.  

By induction, we know that 

$$\Vol_n \phi^{(k)}(X_j) \le \Vol_n \phi^{(k)}(X) \cap F_j \le 2 C \epsilon r_1(F_j)^n e^{- \beta k}.$$

\noindent The pushout lemma for small surfaces tells us that $\tilde \phi^{(k)}(X_j)$ lies in the $w_j$-neighborhood of $\partial K_j$ where
$\sigma_n w_j^n = 2 C \epsilon r_1(F_j)^n e^{- \beta k}$.  Since we can choose $K_j$ as close as we like to the whole
face $F_j$, we see that $\tilde \phi^{(k)}(X) \cap F_j$ lies in the $W_j$-neighborhood of $\partial F_j$ for 
$\sigma_n W_j^n = 3 C \epsilon r_1(F_j)^n e^{- \beta k}$.  Rearranging this formula, we
get the following inequality for $W_j$.

$$W_j / r_1(F_j) \le [3 C \sigma_n^{-1} \epsilon
e^{-\beta k}]^{1/n}. \eqno{(1)}$$

If $\phi^{(k)}(x)$ lies in a face $F$, then $\tilde \phi^{(k)} (x)$ lies in the same face $F$, and so $\tilde \phi^{(k)}$ is subordinate
to the cover.

We define a map that pulls a small
neighborhood of the (k-1)-skeleton of $N$ into the
(k-1)-skeleton.  Our map will be called $R_\delta$, and it
depends on a number $\delta$ in the range $0 < \delta < 1/2$. 
The basic map is a map from an interval $[0, r]$ to itself, which
takes the set $[0, \delta r]$ to $0$, and the set $[r- \delta r,
r]$ to $r$, and linearly stretches the set $[\delta r, r - \delta
r]$ to cover $[0, r]$.  The Lipschitz constant of this map is $(1
- 2 \delta)^{-1}$.  Now we apply this map separately to each
coordinate $\phi_i$ of the rectangle $\prod_{i=1}^D [0,r_i]$.  The resulting map
is $R_\delta$.

Our map $\phi^{(k-1)}$ will be $R_{\delta(k)} \circ \tilde \phi^{(k)}$
for a well-chosen $\delta(k)$.  Notice that for any $\delta$, the map
$R_\delta$ sends each open face $F$ into $\bar F$.  By Lemma \ref{subcheck},
$R_\delta \circ \tilde \phi^{(k)}$ is subordinate to the cover.

Also, for sufficiently big $\delta$, $R_\delta \circ \tilde \phi^{(k)}$ maps
$X$ into $N^{(k-1)}$.  In particular, if
$\delta \ge [3 C \sigma_n^{-1} \epsilon e^{-\beta k}]^{1/n}$, then
inequality $(1)$ guarantees that for each k-face $F_j$, $W_j / r_1(F_j) \le \delta$.  
Since $\tilde \phi^{(k)} \cap F_j$ lies in the $W_j$-neighborhood of $\partial F_j$,  
$R_\delta \circ \tilde \phi^{(k)}$ lies in the
(k-1)-skeleton of $N$.  We define $\delta(k) := [3 C \sigma_n^{-1}
\epsilon e^{- \beta k}]^{1/n}$, and then we define $\phi^{(k-1)} =
R_{\delta(k)} \circ \tilde \phi^{(k)}$.

To close the induction, we just need to check the estimate $(**)$ for $\phi^{(k-1)}$.
In other words, we need to prove that

$$\Vol_n [\phi^{(k-1)}(X) \cap \Star(F)] < 2 C \epsilon r_1(F)^n e^{- \beta d(F)}. $$

By Lemma \ref{Gu2lemma}, we know that

$$\Vol_n [\phi^{(D)}(X) \cap \Star(F)] \le C \epsilon r_1(F)^n e^{- \beta d(F)}. $$

We know that the pushout construction does not increase any volumes.  In particular,
for any $k \le l \le D$, we know that in every face $F
\subset N$, $\Vol_n [\tilde \phi^{(l)}(X) \cap F] \le \Vol_n [\phi^{(l)}(X) \cap F]$.  Hence we know that

$$\Vol_n [ \tilde \phi^{(l)}(X) \cap \Star(F)] \le \Vol_n[ \phi^{(l)}(X) \cap \Star(F)]. $$

On the other hand, the map $R_\delta$ does not increase volumes very much
if $\delta$ is small.  The Lipschitz constant of $R_{\delta}$ is $[1 - 2 \delta]^{-1}$.
 Therefore, for any $n$-dimensional surface, $\Sigma$, 
 
$$ \Vol_n [R_\delta(\Sigma)] \le [1 - 2 \delta]^{-n} \Vol_n [\Sigma].$$

Also, we have seen that $R_\delta$ maps each face $F$ into $\bar F$.  Hence
$R_\delta^{-1}(\Star(F)) \subset \Star(F)$.  Therefore we get the following inequality.

$$\Vol_n [  \phi^{(l-1)}(X) \cap \Star(F) ]  \le [1 - 2 \delta(l)]^{-n} \Vol_n[ \tilde \phi^{(l)} (X) \cap \Star(F)] \le $$

$$ \le [1 - 2 \delta(l)]^{-n} \Vol_n [\phi^{(l)} (X) \cap \Star(F)] . $$

By combining these inequalities for all $k \le l \le D$, and using inequality \ref{volestD} for $\phi^{(D)}$, we get
the following inequality.

$$ \Vol_n [\phi^{(k-1)}(X) \cap \Star(F)] \le \prod_{l=k}^D [1 - 2 \delta(l)]^{-n} C \epsilon r_1(F)^n e^{- \beta d(F)}. $$

Therefore, it suffices to prove that $\prod_{l=k}^D [1 - 2 \delta(l)]^{-n} \le \prod_{l=n}^\infty [1 - 2 \delta(l)]^{-n} < 2$.  Recalling
the definition of $\delta(l)$, it suffices to prove that

$$\prod_{l=n}^\infty \left[ 1 - 2 (3 C \sigma_n^{-1} \epsilon e^{- \beta l})^{1/n } \right]^{-n} < 2. $$

This formula is a little messy, but $C, \sigma,$ and $\beta$ are just dimensional constants.  Because of the
exponential decay $e^{- \beta l}$, the infinite product converges.  And we can choose $\epsilon$ sufficiently
small so that the product is less than 2.

This finishes the proof of the main theorem, except for the proof of the pushout lemma for small surfaces.  We
prove the pushout lemma in Section 3, following preliminary work in Section 2.

\section{An isoperimetric extension lemma}

In order to prove the pushout lemma, we need the following version of the isoperimetric inequality.
Suppose that $X$ is a piecewise smooth $n$-dimensional manifold and $F: X \rightarrow \RR^N$ is a piecewise smooth map.  We write $\Vol_n F(X)$ for the volume of $X$ with the induced Riemannian metric given by pullback with $F$.  Equivalently, we can think of $\Vol_n F(X)$ as the volume of the image of $F$ counted with multiplicity.

\begin{lemma} \label{extlemma} (Extension lemma) For each dimension $n \ge 2$, there exists an ``isoperimetric constant" $I(n)$ so that the following holds.

Suppose that $X$ is a compact piecewise-smooth $n$-dimensional manifold with boundary.  If $F_0$ is a map from $\partial X$ to $\RR^N$, 
then there is an extension $F: X \rightarrow \RR^N$ so that 

$$\Vol_n F(X) \le I(n) [ \Vol_{n-1} F_0(\partial X) ]^{\frac{n}{n-1}}. $$

\end{lemma}

This lemma can be thought of as a (minor) generalization of the isoperimetric inequality.  For any
integral cycle $y^{n-1} \subset \mathbb{R}^N$, the Michael-Simon isoperimetric inequality says that there is a chain $x^n$ with $\partial x = y$ and $Vol_n(x) \le C_n Vol_{n-1}(y)^\frac{n}{n-1}$ \cite{MS}.  Our extension lemma is a version of this inequality for maps instead of chains.  

Recently, Wenger gave a short constructive proof of the Michael-Simon isoperimetric inequality \cite{W}.  
Our proof of the extension lemma closely follows Wenger's proof.  (Wenger's proof works more generally in Banach spaces, and our extension lemma also generalizes to Banach spaces, but we don't pursue it here.)

\begin{proof}[Proof of extension lemma]  We write $A \lesssim B$ if $A \le C(n) B$.  The constants do not depend on the ambient dimension $N$.

We begin with a cone-type inequality which allows us to construct extensions when the diameter of $F_0(\partial X)$ is not too big.

\begin{lemma} \label{coneineq} (Cone inequality) 
Suppose that $X$ is a compact piecewise-smooth $n$-dimensional manifold with boundary.  If $F_0$ is a map from $\partial X$ to ball of radius $R$, $B^N(R) \subset \RR^N$, then there is an extension $F: X \rightarrow B^N(R)$ so that 

$$\Vol_n F(X) \le C_n R [ \Vol_{n-1} F_0(\partial X) ]. $$

\end{lemma}

\begin{proof} There is a tubular neighborhood $E$ of $\partial X$ in $X$ which is diffeomorphic to $\partial X \times (-1, 1]$.  We choose coordinates $(x,t)$ on this neighborhood, where $x \in \partial X$ and $t \in [0, 1)$.  The boundary $\partial X$ is given by the equation $t=0$.  Let $\rho(t)$ be a smooth non-negative function on $[0,1)$ with $\rho(0) = 1$ and $\rho(t) = 0$ for $t \ge 1/2$.  Then define $F$ on this tubular neighborhood by $F(x,t) = \rho(t) F_0 (x)$.  Note that $F$ maps the subset of $E$ where $1/2 \le t < 1$ to the origin.  Therefore, we can smoothly extend $F$ to all of $X$ by mapping the complement of $E$ to the origin.

The image of $F$ is the cone over the image of $F_0$.  By standard calculations in Riemannian geometry, we get 

$$ \Vol_n F(X) \le (1/n) R \Vol_{n-1} F_0(\partial X). $$

(If we were working in a Banach space, then the constant $1/n$ would have to be replaced by a larger constant $C_n$.)
\end{proof}

(Notice that cone inequality proves the extension lemma in the special case that $Vol_{n-1} F_0(\partial X) = 0$.  So in the rest of the proof, we may assume that
$Vol_{n-1} F_0(\partial X) > 0$.)

Our argument will be by induction on $n$.  The base case is $n=2$, which we now discuss.  In this case, $\partial X$ is a 1-dimensional
manifold.  It consists of finitely many connected components, which we call $\partial X_1, \partial X_2,$ etc.  For each connected component,
the diameter of $F_0(\partial X_i)$ is at most the length of $F_0(\partial X_i)$, which we denote by $L_i$.  For each $i$, we choose a point $y_i$ in the image of $F_0(\partial X_i)$ in $\RR^N$.  We observe that $F_0 (\partial X_i)$ is contained in $B^N(y_i, L_i)$.  

We let $E_i$ be a neighborhood of $\partial X_i$ in $X$, diffeomorphic to $\partial X_i \times [0,1)$.  By the same construction as the cone inequality, we can extend $F_0$ to a map $F: E_i \rightarrow B^N(y_i, L_i)$ so that $F$ maps all but a compact subset of $E_i$ to $y_i$ and so that $\Vol_2 F(E_i) \le C L_i^2$.  Now consider $X' = X \setminus \cup_i E_i$.  We have to define $F$ on $X'$ so that $F$ maps $\partial E_i$ to $y_i$.  We can choose $F$ so that the image of $F$ is a tree, and $\Vol_2 F(X') = 0$.  We have now defined an extension $F$ on $X$ with 

$$ \Vol_2 F(X) \le C \sum_i L_i^2 \le C (\sum_i L_i)^2 = C ( \Vol_1 F_0(\partial X) )^2. $$

This finishes the proof of the extension lemma in the base case $n=2$.  

Now we begin the proof of the inductive step.  We assume that the extension lemma holds in dimension $n-1$.  We construct the extension $F$ by repeatedly using
the following partial extension lemma. 

\newtheorem*{partext}{Partial extension lemma}

\begin{partext} Suppose that $X$ is a compact piecewise-smooth $n$-dimensional manifold with boundary.  If $F_0$ is a map from $\partial X$ to $B^N(R) \subset \RR^N$, then we can decompose $X$ as a union of two $n$-dimensional submanifolds with $\partial$, $X = X_1 \cup G_1$, and we can extend $F_0$ to a map $F_1: G_1 \rightarrow B^N(R)$ so that the following estimates hold.  

$$ \Vol_n F_1(G_1) \lesssim \Vol_{n-1} F_0(\partial X)^{\frac{n}{n-1}}. $$

Next, we note that $\partial X_1 \subset \partial G_1 \cup \partial X$.  Therefore, $F_1$ is defined on $\partial X_1$.  

$$ \Vol_{n-1} F_1(\partial X_1) \le (1 - \delta_n) Vol_{n-1} F_0( \partial X) . $$

\end{partext}

Using the partial extension lemma and the cone inequality, we can quickly finish the proof of the extension lemma.  Suppose that $F_0: \partial X \rightarrow \RR^N$.  Choose a large radius $R$ so that $F_0 (\partial X) \subset B^N(R)$.  
We use the partial extension lemma to define $X_1, G_1, F_1$.  
Then we consider $F_1: \partial X_1 \rightarrow B^N(R)$, and
we apply the extension lemma to it.  We use the partial extension lemma $J$ times, where $J$ is a large number that we will choose below.  We get a sequence of subsets $X_J \subset X_{J-1} \subset ... \subset X_1 \subset X_0 = X$, and a sequence of maps $F_j: X_{j-1} \setminus X_{j} \rightarrow B^N(R)$.  The maps $F_j$ fit together to define a single piecewise smooth map $F: X \setminus X_J \rightarrow B^N(R)$, extending $F_0: \partial X \rightarrow B^N(R)$.

By the second estimate in the partial extension lemma, we know that

$$ \Vol_{n-1} F_j (\partial X_j) \le (1 - \delta_n) \Vol_{n-1} F_{j-1} (\partial X_{j-1}). $$

Therefore,

$$ \Vol_{n-1} F_j (\partial X_j) \le (1 - \delta_n)^j \Vol_{n-1} F_0 (\partial X). $$

Next we can bound the volume of $F_j (X_{j-1} \setminus X_j)$, using the first estimate in the partial extension lemma.  

$$ \Vol_n F_j (X_{j-1} \setminus X_j) \lesssim \left( \Vol_{n-1} F_{j-1} (\partial X_{j-1}) \right)^{\frac{n}{n-1}} \lesssim (1 - \delta_n)^{\frac{n}{n-1} j} \Vol_{n-1} F_0 (\partial X)^\frac{n}{n-1}. $$

Summing the exponential sum, we see that

$$ \Vol_n F(X \setminus X_J) \lesssim \Vol_{n-1} F_0 (\partial X)^\frac{n}{n-1}. $$

This estimate holds uniformly in the choice of $J$.  Finally, we extend $F$ to $X_J$ using the cone inequality.  We get

$$ \Vol_n F(X_J) \lesssim R \Vol_{n-1} F_J (\partial X_J) \lesssim R (1 - \delta_n)^J \Vol_{n-1} F_0 (\partial X). $$

We now choose $J$ sufficiently large in terms of $R$ and $\Vol_{n-1} F_0 (\partial X)$ so that this final term is dominated by the previous term.  

This finishes the proof of the extension lemma from the partial extension lemma.  Now we turn to the proof of the partial extension lemma.

\begin{proof}[Proof of partial extension lemma]

At this point, it is convenient to know that our mapping is an embedding.  To accomplish this, we add extra dimensions to the
target space $\mathbb{R}^N$.  We let $F_0^+: \partial X \rightarrow \mathbb{R}^N \times \mathbb{R}^E$ be given by $F_0$ in the first
factor and by a nice embedding in the second factor.  By scaling the second factor, we can assume that $\Vol_{n-1} F_0^+ (\partial X) \le (1 + \epsilon)
\Vol_{n-1} F_0(\partial X)$ and we can assume that $F_0^+(\partial X) \subset B^{N+E}(R)$.  Now $F_0^+$ is an embedding.  Our construction will give a partial extension $F_1^+: G_1 \rightarrow B^{N+E}(R)$, obeying good estimates.  Finally, we define $F_1$ to be $F_1^+$ composed with the projection from $B^{N+E}(R)$ to $B^N(R)$.  This projection can only decrease volumes, so the resulting partial extension $F_1: G_1 \rightarrow B^N(R)$ obeys the same estimates as $F_1^+$.  In summary, it suffices to consider the case that $F_0$ is an embedding.

The proof of the partial extension lemma is based on a ball-covering argument, the extension lemma in dimension $n-1$, and the cone inequality.  It closely follows the argument in \cite{W}.

We consider the volumes $\Vol_{n-1} [F_0(\partial X) \cap B^N(p, r)]$ for various balls $B^N(p,r) \subset \mathbb{R}^N$.  Because $F_0$ is an embedding, we know that for every regular point $p$ in the image $F_0(\partial X)$, the volume $\Vol_{n-1} [F_0(\partial X) \cap B(p,r)] \ge c_n r^{n-1}$
for all sufficiently small $r$.  (We may take $c_n$ to be one half the volume of the unit $(n-1)$-ball.)  Since $F_0$ is a piecewise smooth embedding, almost every point of the image is regular.  

Fix any regular point $p$.  Let $V(r)$ denote $\Vol_{n-1} [F_0(\partial X) \cap B^N(p, r)]$.  Let $\epsilon_n > 0$ be a small constant that we will choose
below.  We let $r_0$ denote the largest radius $r$ so that $V(r) \ge \epsilon_n r^{n-1}$.  Since $V(r) \le Vol F_0(\partial) < \infty$, such an
$r$ exists, and we have $V(r_0) = \epsilon_n r_0^{n-1}$.  Since $p$ is a regular point, $r_0 > 0$.  By the definition of $r_0$, we see that $V(5 r_0) < \epsilon_n (5 r_0)^{n-1}$.

Now we consider the intersections $F_0(\partial X) \cap \partial B(p,r)$.  Since $F_0$ is in general position, for almost every $r$, $F_0$
is transverse to $\partial B(p,r)$.  By the coarea inequality, $\int_{r_0}^{5 r_0} \Vol_{n-2} [F_0(\partial X) \cap \partial B(p,r)] \le V(5 r_0)$.
Therefore, we can choose a generic value $r \in (r_0, 5 r_0)$ so that $\Vol_{n-2} F_0(\partial X) \cap \partial B(p,r) \lesssim \epsilon_n r_0^{n-2} \le \epsilon_n r^{n-2}$.  We call this radius $r$ a good radius, and we call $B(p,r)$ a good ball.

The good ball $B(p,r)$ obeys three important geometric estimates.

$$ 1. \Vol_{n-1} [ F_0(\partial X) \cap B(p,r) ] \gtrsim \epsilon_n r^{n-1}. $$

$$ 2. \Vol_{n-1} [F_0(\partial X) \cap B(p, 5r) ] \lesssim \epsilon_n r^{n-1}. $$

$$ 3. \Vol_{n-2} [F_0(\partial X) \cap \partial B(p,r) ] \lesssim \epsilon_n r^{n-2}. $$

Using the Vitali covering lemma, we choose a finite collection of disjoint good balls $B_i = B(p_i, r_i)$ so that $\cup B(p_i, 5r_i)$ covers most of $F_0 (\partial X)$.  More precisely:

$$\Vol_{n-1} \left[ F_0(\partial X) \cap (\cup B(p_i, 5r_i)) \right] \ge \frac{1}{2} \Vol_{n-1} F_0 (\partial X). $$

By properties 1 and 2, it follows that 

\begin{equation} \label{vitcov} \Vol_{n-1} \left[ F_0(\partial X) \cap (\cup_i B_i) \right] \gtrsim \Vol_{n-1} F_0(\partial X). \end{equation} 

Now we are ready to define the sets $G_1$ and $X_1$.  We first extend $F_0$ to $X$ in an arbitrary (generic) way.   We denote this extension by $F_0$.  We define $G_1 \subset X$ to be $F_0^{-1} ( \cup_i B_i)$, and we define $X_1 \subset X$ to be the complement of $G_1$.  Since we chose the radii $r_i$ generically, we can assume that $F_0$ is transverse to $\partial B_i$, and so $G_1$ and $X_1$ are (piecewise smooth) manifolds with boundary.  

We still have to define the partial extension $F_1: G_1 \rightarrow \mathbb{R}^N$, extending $F_0$.   We let $G_1(i) := F_0^{-1}(B_i)$.  The boundary of $G_1(i)$ has two parts.  One part of $\partial G_1(i)$ lies in $\partial X$.  We let $Y_1(i)$ denote the rest of the boundary:

$$ Y_1(i) := \partial G_1(i) \setminus \partial X. $$

We can think of $Y_1(i)$ as the set of $x$ in the interior of $X$ so that $F_0(x) \in \partial B_i$.  Now $Y_1(i)$ is an $(n-1)$-dimensional piecewise smooth manifold with boundary.  Its bounday lies in $\partial X$, and $F_0: \partial Y_1(i) \rightarrow \partial B_i$.  By Property 3 above, we know that 

$$\Vol_{n-2} F_0 (\partial Y_1(i) ) = \Vol_{n-2} [F_0(\partial X) \cap \partial B(p_i,r_i) ] \lesssim \epsilon_n r_i^{n-2}. $$

Now we can use the extension lemma for dimension $n-1$ to define a good map $F_1$ on $Y_1(i)$ extending $F_0: \partial Y_1(i) \rightarrow B(p_i, r_i)$.  The map $F_1$ will obey the following estimate:

$$ \Vol_{n-1} F_1( Y_1(i) ) \lesssim \Vol_{n-2} F_0(\partial Y_1(i))^{\frac{n-1}{n-2}} \lesssim I(n-1) \epsilon_n^{\frac{n-1}{n-2}} r_i^{n-1}. $$

\noindent (Here $I(n-1)$ denotes the constant in the extension lemma in dimension $n-1$.)

We have now defined $F_1$ on all of $\partial X_1$, and we are ready to prove the crucial volume estimate $\Vol_{n-1} F_1( \partial X_1) \le (1 - \delta_n) Vol_{n-1} F_0(\partial X)$.  
To do this, let $\partial X(i) := \partial X \cap F_0^{-1}(B_i)$.  Now $\partial X_1$ is formed from $\partial X$ by deleting each $\partial X(i)$ and adding in each $Y_1(i)$.  Therefore,

$$ \Vol_{n-1} F_1 (\partial X_1) = \Vol_{n-1} F_0(\partial X) - \sum_i \Vol_{n-1} F_0(\partial X) \cap B_i + \sum_i \Vol_{n-1} F_1( Y_1(i)). $$

By the first property of a good ball, we know that $\Vol_{n-1} F_0 (\partial X) \cap B_i \gtrsim \epsilon_n r_i^{n-1}$.  On the other
hand, $\Vol_{n-1} F_1 (Y_1(i)) \lesssim I(n-1) \epsilon_n^{\frac{n-1}{n-2}} r_i^{n-1}$.  The key observation is that we have a better power of $\epsilon_n$ in the volume bound for $F_1 (Y_1(i))$.  At this point we choose $\epsilon_n$ sufficiently
small compared to the other dimensional constants to guarantee that

\begin{equation} \label{volF_1Y} \Vol_{n-1} F_1(Y_1(i) ) \le \frac{1}{2} \Vol_{n-1} F_0(\partial X) \cap B_i. 
\end{equation}

Plugging this estimate in, we see that

$$ \Vol_{n-1} F_1 (\partial X_1) \le \Vol_{n-1} F_0(\partial X) - \frac{1}{2} \sum_i \Vol_{n-1} F_0(\partial X) \cap B_i . $$

But as we noted above in Equation \ref{vitcov}, the first two properties of a good ball imply that
imply that $\sum_i \Vol_{n-1} F_0(\partial X) \cap B_i \gtrsim \Vol_{n-1} F_0(\partial X)$.  Therefore, we conclude

$$ \Vol_{n-1} F_1(\partial X_1) \le (1 - \delta_n) \Vol_{n-1} F_0(\partial X). $$

This is one of the two estimates in the conclusion of the partial extension lemma.  

Recall that $G_1$ is defined to be $X \setminus X_1$.  In other words, $G_1 = F_0^{-1} (\cup_i B_i)$.  We let $G_1(i)$ be $F_0^{-1}(\bar B_i)$.  Next we have to define $F_1$ on $G_1$ and bound the volume $\Vol_n F_1(G_1)$.  We will use the cone inequality to define $F_1$ on each $G_1(i)$.  

Note that $\partial G_1(i) \subset Y_1(i) \cup \partial X$.  We have already defined $F_1$ on $Y_1(i)$, and $F_1 = F_0$ on $\partial X$, and so we've already defined $F_1$ on $\partial G_1(i)$.  We also have estimates about the volume of $F_1 (Y_1(i))$ and $F_1( \partial X \cap G_1(i))$.   

Using Property 2 of good balls, we see that 

$$ \Vol_{n-1} F_1 (\partial X \cap G_1(i)) = \Vol_{n-1} F_0 (\partial X) \cap B(p_i, r_i) \lesssim r_i^{n-1}. $$

And by Equation \ref{volF_1Y}, we know that

$$ \Vol_{n-1} F_1 (Y_1(i)) \le \frac{1}{2}  \Vol_{n-1} F_0 (\partial X) \cap B(p_i, r_i) \lesssim r_i^{n-1}. $$

Altogether, we see that $\Vol_{n-1} F_1 ( \partial G_1(i) ) \lesssim r_i^{n-1}$.  By the cone inequality, we can extend $F_1$ to all of $G_1(i)$ so that

$$ \Vol_n F_1( G_1(i) ) \lesssim r_i  \Vol_{n-1} F_1 (\partial G_1(i)) \lesssim r_i^n \sim [ \Vol_{n-1} F_0(\partial X) \cap B_i ]^\frac{n}{n-1}. $$

Adding the contributions from different balls $B_i$, we see that

$$ \Vol_n F_1(G_1) \lesssim \sum_i [ \Vol_{n-1} F_0(\partial X) \cap B_i ]^\frac{n}{n-1} \lesssim Vol_{n-1} F_0(\partial X)^{\frac{n}{n-1}}. $$

We have now defined the partial extension $F_1: G_1 \rightarrow \RR^N$, and we have proven both estimates about $F_1$.

(We would like the map $F_1$ to send $G_1$ to $B^N(R)$.  The construction above may not map $G_1$ to $B^N(R)$, but we can fix that problem in a simple way.  Let $\pi: \RR^N \rightarrow B^N(R)$ be the closest point map.  (The map $\pi$ is the identity on $B^N(R)$, and it maps each point outside of $B^N(R)$ to the closest point on the boundary of $B^N(R)$.)  Since balls are convex, $\pi$ is distance-decreasing, and so $\pi$ decreases $k$-dimensional volumes for all $k$.  The map $\pi \circ F_1$ maps $G_1$ to $B^N(R)$, and it obeys all the same estimates as $F_1$.)

\end{proof}

This finishes the proof of the extension lemma.

\end{proof}

\section{Proof of the push-out lemma for small surfaces}

In this section, we prove Lemma \ref{pushoutlemma}.  

We explained in the introduction that there is some intuition for the pushout lemma coming from the monotonicity formula in minimal surface theory.  
We will prove the push-out lemma without using minimal surface theory, but the proof
is based on some version of the monotonicity idea.  We intersect $\phi_0(X)$ with various balls $B(p,r) \subset K$.  If the volume
of the intersection does not obey a monotonicity-type estimate, we can homotope $\phi_0$ to a new map with smaller volume.  
This homotopy is accomplished using the extension lemma from the last section.  We continue performing this type of homotopy,
decreasing the total volume each time, until the amount of volume outside the $W$-neighborhood of $\partial K$ gets as tiny
as we like.  Finally, we push this tiny volume into the $W$-neighborhood of $\partial K$ using the Federer-Fleming push-out lemma.

Now we turn to the details.  We homotope $\phi_0$ to $\phi_1$ in a sequence of small steps, and each step is given by the
following lemma.

For any $W > 0$, we let $K(W)$ denote the set $\{ x \in K | \Dist(x, \partial K) > W \}$.

\begin{lemma} \label{incrpushout} Let $\delta > 0$ be any number.  Suppose that $\Phi: (X, \partial X) \rightarrow (K, \partial K)$.  Let $V = Vol_n \Phi(X)$.  
Define $W$ so that $ \sigma_n W^n = V$.  Suppose that the volume of $\Phi(X) \cap K(W)$ is at least $\delta$.  Then we can homotope $\Phi$ rel
$\partial X$ to a new map $\Phi'$ so that its total volume is decreased by a definite amount:

$$ \Vol_n [ \Phi'(X) ] \le \Vol_n [ \Phi(X) ] - c( \delta, n, K) .$$

\end{lemma}

\begin{proof} The constant $\sigma_n$ in the formula above is the same constant as in our push-out lemma for small surfaces.  We haven't
chosen the constant yet.  We will choose a sufficiently small $\sigma_n > 0$ below.

We find a point $p \in K(W)$ so that

$$ \Vol_n [ \Phi(X) \cap B(p, W/2) ] \ge c(\delta, n, K) . $$

We can find $p$ by averaging over all points in $K(W)$.  The total volume of $\Phi(X) \cap K(W)$
is assumed to be at least $\delta$, and so we can find a ball that contains a certain definite amount
of volume $c(\delta, n , K)$.

We let $V(R)$ denote the volume of $\Phi(X) \cap B(p, R)$.  We let $A(R)$ denote the area of $\Phi(X) \cap \partial B(p,R)$.  

Our argument will be based on the extension lemma from the last section.  We recall that $I(n)$ is the isoperimetric constant in the extension lemma.

We call a radius $R$ good if $W/2 < R < W$ and if $V(R)$ obeys the inequality

$$ V(R) \ge 2 I(n) A(R)^{\frac{n}{n-1}} . $$

For generic $R$, $\Phi^{-1}[\partial B(p,R)] \subset X$ is a piecewise smooth (n-1)-manifold, and so $\Phi^{-1} [B(p,R)] = X' \subset X$ is a piecewise smooth manifold with boundary (of the same dimension $n$ as $X$).  Note that $\Vol_{n-1} \Phi(\partial X') = A(R)$.  
 Also, since $p \in K(W)$ and $R < W$, the ball $B(p,R)$ is contained in the interior of $K$, and so $X'$ does not intersect $\partial X$.

If there is a generic good radius $R$, then we let $B := B(p,R)$, and we apply the extension lemma to the map $\Phi: (X', \partial X') \rightarrow (B, \partial B)$.
The extension lemma gives us a new map $\Phi': (X', \partial X') \rightarrow (B, \partial B)$ so that $\Phi'$ agrees with $\Phi$ on $\partial X'$ and we get
the following volume estimate:

$$ \Vol_n \Phi'(X') \le I(n) A(R)^{\frac{n}{n-1}} \le \frac{1}{2} V(R) . $$

Now we extend $\Phi'$ to all of $X$ by letting $\Phi'$ agree with $\Phi$ on $X \setminus X'$.  Since $K$ is convex, we can easily homotope $\Phi$ to $\Phi'$ rel $\partial X$.  Most importantly, the volume $\Phi'(X)$
is smaller than the volume of $\Phi(X)$ by a definite amount:

$$ \Vol_n \Phi(X) - \Vol_n \Phi'(X) = \Vol_n \Phi(X') - \Vol_n \Phi'(X') \ge \frac{1}{2} V(R) \ge 1/2 \Vol_n [ \Phi(X) \cap B(p, W/2) ] \ge c(\delta, n, K) . $$

This estimate suffices to prove the Lemma, provided that we can find a generic good radius $R$.  Suppose that almost every
radius $R$ in the range $W/2 < R < W$ is not good.  In other words, we have the following inequality.

$$ V(R) \le 2 I(n) A(R)^{\frac{n}{n-1}}, \textrm{ for  almost every } R \in [W/2, W].$$

But $V'(R) \ge A(R)$.  Hence we have

$$V'(R) \ge [ 2 I(n) ]^{- \frac{n-1}{n}} V(R)^{\frac{n-1}{n}}  \textrm{ for almost every } R \in [W/2, W].$$

Equivalently,

$$ \frac{d}{dR} [ V(R)^{1/n} ] = (1/n) V'(R) V(R)^{- \frac{n-1}{n}} \ge (1/n) [2 I(n)]^{- \frac{n-1}{n}} . $$

Integrating this inequality from $W/2$ to $W$, we see that

$$ V(W) \ge C(n) W^n, $$

\noindent where $C(n) = n^{-n} 2^{-(n-1)} I(n)^{-(n-1)}$.  But $V(W) \le V = \sigma_n W^n$.   If we choose $\sigma_n > 0$ sufficiently small, we get a contradiction.  Hence there exists a generic good radius $R$, and the lemma is proved.  \end{proof}

With this lemma, we can prove the push-out lemma for small surfaces.

\begin{proof}[Proof of the push-out lemma for small surfaces]

We fix a $\delta > 0$ and we use Lemma \ref{incrpushout} repeatedly.  We get a sequence of maps $\psi_j$ homotopic to $\phi_0$ rel $\partial X$, and with $\Vol_n \psi_j(X) \cap K(W)$ decreasing steadily until $\Vol_n \psi_j(X) \cap K(W) < \delta$.   We label this last map $\phi_{1/2}$.  We know that $\phi_{1/2}$ is homotopic to $\phi_0$ rel $\partial X$ and 
that $\Vol_n \phi_{1/2}(X) \le \Vol_n \phi_0(X)$ and the volume of $\phi_{1/2}(X) \cap K(W)$ is less than $\delta$.  

Next, we use the Federer-Fleming push-out lemma to remove the tiny volume from $K(W)$.  Note that $K(W)$ is itself a convex set.  We define $X' = \phi_{1/2}^{-1}(K(W))$.  We note that $\phi_{1/2}: (X', \partial X') \rightarrow K(W)$.  Using Lemma \ref{ffpushout}, we homotope $\phi_{1/2}$ rel $\partial X'$ to a new map $\phi_1: X' \rightarrow \partial K(W)$ so that $\Vol_n \phi_1(X') \le C(n, N, K) \delta$.  Now we get a new map $\phi_1$ from $X$ to $K$, with image lying in the $W$-neighborhood of $\partial K$, and with $\Vol_n \phi_1(X) \le \Vol_n \phi_0(X) + C(n,N, K(W)) \delta$.  Since we can choose $\delta$ as small as we like, we can arrange that $\Vol_n \phi_1(X)$ is as close as we like to $\Vol_n \phi_0(X)$.

In fact, with a little more work, we can arrange that $\Vol_n \phi_1(X) \le \Vol_n \phi_0(X)$.  This is not an important point, but it is convenient for keeping the notation simple when we apply the pushout lemma.  We consider two cases.  If $\Vol_n \phi_0(X) \cap K(W) = 0$, then we just apply the Federer-Fleming pushout as in the last paragraph to homotope $\phi_0$ to $\phi_1$ with image in the $W$-neighborhood of $\partial K$.  If $\Vol_n \phi_0(X) \cap K(W) = V_0 > 0$, then we apply Lemma \ref{incrpushout}.  The first application gives a map $\phi_{1/4}$ where $\Vol_n \phi_{1/4}(X) \le \Vol_n (\phi_0(X)) - c(n, K, V_0)$.  After noting the constant $c(n, K, V_0) > 0$, we now choose some $\delta > 0$, and we continue to apply Lemma \ref{incrpushout} until we arrive at a map $\phi_{1/2}: X \rightarrow K$ so that $\Vol_n \phi_{1/2}(X) \cap K(W) < \delta$.  At this point, we use the Federer-Fleming pushout argument to homotope $\phi_{1/2}$ to a map $\phi_1$ with image in the $W$-neighborhood of $K$.  The volume of $\phi_1(X)$ is at most

$$\Vol_n \phi_1(X) \le \Vol_n (\phi_0(X)) - c(n, K, V_0) + c(n, N, K(W)) \delta.  $$

After noting $c(n, K, V_0)$, we choose $\delta$ sufficiently small so that $\Vol_n \phi_1(X) \le \Vol_n \phi_0(X)$.  

\end{proof}

\section{An example related to Uryson width}

In this section, we give an example of a Riemannian manifold where each unit ball has small Uryson width,
and yet the whole space has large Uryson width.

\begin{prop} For any $\epsilon > 0$, there is a metric $g_\epsilon$ on $S^3$ so that the following holds.
Every unit ball in $(S^3, g_\epsilon)$ has Uryson 2-width $< \epsilon$.  The whole manifold $(S^3, g_\epsilon)$
has Uryson 2-width at least 1.
\end{prop}

The main idea of the proof comes from Example $(H_1'')$ in \cite{Gr2}.

\begin{proof} Let $g_0$ denote the standard unit-sphere metric on $S^3$.  Let $T$ denote a fine triangulation
of $(S^3, g_0)$.  We choose $T$ sufficiently fine so that the lengths of edges are at most $\delta$, for a small
constant $\delta > 0$ we may choose later.  We let $K_1$ denote the 1-skeleton of the triangulation, and we let
$K_2$ denote the dual 1-skeleton.  There is one vertex of $K_2$ in each 3-face of $T$, and there is an edge
connecting two vertices of $K_2$ if the corresponding 3-faces share a common 2-face in their boundaries.

Now we let $U_1$ denote the $\delta_1$-neighborhood of $K_1$ for a small constant $\delta_1 << \delta$.
We let $U_2$ denote $S^3 \setminus \bar U_1$ and we let $\Sigma$ denote $\partial U_1$.  By choosing
$\delta_1$ sufficiently small, we can arrange that $K_2 \subset U_2$.

Next we choose retractions $r_1: \bar U_1 \rightarrow K_1$, and $r_2: \bar U_2 \rightarrow K_2$.
We can choose $r_1$ and $r_2$ to obey the following diameter estimates:

\begin{itemize} 

\item For each $y \in K_i$, the diameter of $r_i^{-1}(y)$ within $\bar U_i$ is $\lesssim \delta$.

\item For each $y \in K_i$, the diameter of $r_i^{-1}(y) \cap \Sigma$ within $\Sigma$ is $\lesssim \delta$.

\end{itemize}

We define the space $X$ to be $U_1 \cup (\Sigma \times [0, 10]) \cup U_2$, where the three pieces are glued
together as follows.  We glue $\partial U_1 = \Sigma$ to $\Sigma \times \{ 0 \}$ using the identity map, and
we glue $\partial U_2 = \Sigma$ to $\Sigma \times \{10 \}$ using the identity map.  We put a metric on $X$, where
$U_1$, $U_2$, and $\Sigma$ have the metric inherited from $(S^3, g_0)$, and $\Sigma \times [0,10]$ has the product
metric.  The space $X$ is homeomorphic to $S^3$, and it is $(1+ \delta)$ bilipschitz to a Riemannian metric $(S^3, g)$.

The Riemannian metric $(S^3, g)$ is formed by taking a small tubular neighborhood of $\Sigma \subset S^3$,
say $N = \Sigma \times (-\delta_2, \delta_2)$, and stretching it so that the $(-\delta_2, \delta_2)$ direction becomes
long.  As such $UW_2(S^3, g) \ge (1 + \delta)^{-1} UW_2(S^3, g_0) \gtrsim 1$ by the Lebesgue covering lemma.

(The exact value of $UW_2(S^3, g_0)$ is known by work of Katz, see \cite{K}.)

On the other hand, if $B_1 \subset (S^3, g)$ denotes any unit ball, then 
$UW_2 (B_1) \le (1+ \delta) UW_2(B)$ where $B \subset X$ is a ball of radius $(1 + \delta)$.
Any such ball $B \subset X$ lies either in $X \setminus U_1$ or $X \setminus U_2$.  Therefore,
the result follows from the following lemma.

\begin{lemma} If $\delta > 0$ is sufficiently small, 
the Uryson width $UW_2(X \setminus U_i) \lesssim \delta$ for $i = 1, 2$.

\end{lemma}

\begin{proof}[Proof of Lemma] The situation is the same for $i=1,2$.  We write down the 
proof for $i =1$.

We have to find a continuous map $\pi$ from $X \setminus U_1$ to a 2-dimensional complex $Y$
with small fibers.  Our target $Y$ is $K_2 \times [0, 10]$.  The domain $X \setminus U_1$
is equal to $\Sigma \times [0, 10] \cup U_2$.  We let $(x, t)$ be coordinates on $\Sigma \times [0,10]$.
We define $\pi$ on $\Sigma \times [0,10]$ by $\pi(x,t) = ( r_2(x), t) \in K_2 \times [0,10]$.  
We let $x$ denote the coordinate on $\bar U_2$.  We define $\pi$ on $\bar U_2$ by $\pi(x) = ( r_2 (x), 10)
\in K_2 \times [0,10]$.  Since $\Sigma \times [0,10]$ and $U_2$ are glued together by identifying
$\Sigma \times \{ 10 \}$ with $\partial U_2 = \Sigma$, the map $\pi$ is a continuous map from
$X \setminus U_1$ to $Y = K_2 \times [0,10]$.

Next, we estimate the size of the fiber $\pi^{-1}(k,t)$, where $k \in K_2$ and $t \in [0,10]$.  
The estimate has two cases.  If $t < 10$, then the fiber has the form $(r_2^{-1}(k) \cap \Sigma)
\times \{ t \} \subset \Sigma \times [0,10]$.  So the diameter of $\pi^{-1}(k,t)$ is at most the
diameter of $r_2^{-1}(k) \cap \Sigma$ within $\Sigma$, which is $\lesssim \delta$.

If $t = 10$, then the fiber $\pi^{-1}(k,t)$ has the form $r_2^{-1}(k) \subset \bar U_2$.  So the
fiber has diameter at most the diameter of $r_2^{-1}(k)$ in $\bar U_2$, which is also $\lesssim \delta$. \end{proof}

So our metric $(S^3, g)$ has Uryson 2-width $\gtrsim 1$ and yet every ball of radius 1 in $(S^3, g)$ has
Uryson width $\lesssim \delta$.  By taking $\delta$ small and rescaling the metric a little, we get a metric
$(S^3, g_\epsilon)$ as desired.  \end{proof}

Remark:  We also note that this metric $(S^3, g_\epsilon)$ has volume $\sim 1$ and diameter $\sim 1$.

\section{Open problems}

The Szpilrajn theorem holds very generally for all compact metric spaces.  Our theorem is only proven for
Riemannian manifolds.  But I don't know any counterexample to prevent Theorem \ref{main} from
generalizing to compact metric spaces.

\begin{ques} Suppose that $X$ is a compact metric space.  Suppose that each unit ball of $X$ has
n-dimensional Hausdorff measure $< \epsilon_n$.  If $\epsilon_n$ is chosen sufficiently small, does
this imply that $UW_{n-1}(X) \le 1$?
\end{ques}

In fact, something even more general based on the Hausdorff content looks very plausible.  Recall
that the n-dimensional Hausdorff content of a subset $S$ in a metric space $X$ is the infimum of
$\sum_i r_i^n$ among all coverings of $S$ by countably many balls $B(x_i, r_i)$.  

\begin{ques} Suppose that $X$ is a compact metric space, and that each unit ball in $X$ has
n-dimensional Hausdorff content $< \epsilon_n$.  If $\epsilon_n$ is chosen sufficiently small, does
this imply that $UW_{n-1}(X) \le 1$?
\end{ques}

This result would be stronger than our theorem even for Riemannian manifolds.  In particular, it may
apply to a Riemannian manifold $(X^d, g)$ with dimension $d > n$.  When $d > n$, the n-dimensional
Hausdorff {\it measure} of a unit ball in $X^d$ is always infinite.  But the n-dimensional Hausdorff content
of a unit ball in $X^d$ is always finite, and for some $X$ it can be very small.

There are also open questions related to the funny example in Proposition \ref{counterex}. 

\begin{ques} Suppose that $(M^n, g)$ is a Riemannian manifold so that each unit ball $B \subset M$ has
$UW_q(B) < \epsilon$.  If $\epsilon$ is sufficiently small, does this inequality imply anything about $UW_{q'}(M)$
for some $q' \ge q$?
\end{ques}

\section{Appendix: non-compact manifolds and manifolds with boundary}

Our main theorem also holds for compact Riemannian manifolds with boundary.

\begin{corollary} There exists $\epsilon_n > 0$ so that the following holds.  
If $X$ is a compact Riemannian manifold with boundary, and if there is a radius $R$ so that
every ball of radius $R$ in $X$ has volume at most $\epsilon_n R^n$, then $UW_{n-1}(X) \le R$.
\end{corollary} 

\begin{proof} Let $DX$ denote the double of $X$, which is a closed Riemannian manifold.  If $x \in X$, then the ball in $DX$ centered at $x$ of radius $R$ is contained in the double of the ball in $X$ centered at $x$ of radius $R$.  Therefore, every ball of radius $R$ in $DX$ has volume at most $2 \epsilon_n R^n$.  If $\epsilon_n$ is sufficiently small, then Theorem \ref{main} implies that $UW_{n-1}(DX) \le R$.  In other words, there is an (n-1)-dimensional polyhedron $Y$ and a map $\pi: DX \rightarrow Y$ so that each fiber $\pi^{-1}(y)$ has diameter at most $R$ in $DX$.  We restrict $\pi$ to a map $X \rightarrow Y$.   Finally, we note that for two points $x_1, x_2 \in X$, the distance from $x_1$ to $x_2$ in $X$ is equal to the distance in $DX$.  Therefore, the diameter of each fiber $\pi^{-1}(y)$ in $X$ is at most $R$, and we conclude that $UW_{n-1}(X) \le R$.  
\end{proof}

Our main theorem also extends to complete Riemannian manifolds in the following sense: 

\begin{thm} There exists $\epsilon_n > 0$ so that the following holds.  
If $(M^n, g)$ is a complete Riemannian manifold, and if there is a radius $R$ so that
every ball of radius $R$ in $(M^n, g)$ has volume at most $\epsilon_n R^n$, then there is a continuous map $\pi: M \rightarrow Y$ for an infinite $(n-1)$-dimensional complex $Y$ so that each fiber $\pi^{-1}(y)$ has diameter at most $R$ in $(M^n, g)$.  
\end{thm}

The proof is essentially the same as the proof for closed manifolds.  The main tricky issue is that the rectangular nerve $N$ is not finite-dimensional but only locally finite dimensional.  The cover by balls $B_i$ will be locally finite, but the multiplicity of the cover may go to infinity at a sequence of points going to infinity in $M$.  (See Section 6 of \cite{Gu2} for an explanation of how to choose the balls $B_i$ on a complete Riemannian manifold.)  Therefore, each face of $N$ is finite dimensional, and each face of $N$ is adjacent to only finitely many other faces of $N$, but the dimension of the faces may be unbounded.  

In the proof of Theorem \ref{main}, we began with a map $\phi = \phi^{(D)}: M \rightarrow N$, and
we built a sequence of homotopic maps
$\phi = \phi_D \sim \phi_{D-1} \sim ... \sim
\phi_{n-1}$, where $D$ was the dimension of $N$ and $\phi_k$ maps $M$ to the
k-skeleton of $N$.  In general, the dimension of $N$ is not
finite, but is only locally finite, and we must proceed a little
differently.  Instead, we construct an infinite sequence of
maps $\phi_k: M \rightarrow N$ where each $\phi_k$ maps $M$ to the
k-skeleton of $N$ and each $\phi_k$ is subordinate to our cover.  

In a region of $N$ where the dimension is less than $k$, we
define $\phi_k$ to be the infinite composition $R_{\delta(k+1)}
\circ R_{\delta(k+2)} \circ ... $ applied to $\phi$.  (This infinite
composition is defined to be the limit of the maps
$R_{\delta(k+1)} \circ ... \circ R_{\delta(N)}$ as $N$ goes to
infinity.  The sequence of maps converges uniformly on compact
sets.)  In a region where the dimension of $N$ is at least $k$,
we define $\phi_k$ from $\phi_{k+1}$ as in the proof of Theorem \ref{main}.
All the maps $\phi_k$ are subordinate to the cover, and the volume bounds work in the same way
as in the proof of Theorem \ref{main}.


\begin{thebibliography}{5}

\vskip.125in

\bibitem[Gr1]{Gr1} Gromov, M., Filling Riemannian manifolds, J. Differential
Geom. 18 (1983) no. 1, 1-147.

\bibitem[Gr2]{Gr2} Gromov, M., Width and related invariants of Riemannian manifolds. AstŽrisque No. 163-164 (1988), 6, 93-109, 282 (1989). 

\bibitem[Gr3]{Gr3}  Gromov, M., Large Riemannian manifolds, in {\it Curvature
and topology of Riemannian manifolds} (Katata, 1985), 108-121, Lecture
Notes in Math. 1201, Springer, Berlin, 1986.

\bibitem[Gr4]{Gr4} Gromov, M., {\it Metric Structures for Riemannian and
Non-Riemannian Spaces}, Progress in Math. 152, Birkhauser Boston Inc.,
Boston, MA, 1999.

\bibitem[Gu1]{Gu1} Guth, L., Notes on Gromov's systolic estimate. Geom. Dedicata 123 (2006), 113-129. 

\bibitem[Gu2]{Gu2} Guth, L.,  Guth, Volumes of balls in large Riemannian manifolds. Ann. of Math. (2) 173 (2011), no. 1, 51-76.  

\bibitem[Gu3]{Gu3} Guth, L. Metaphors in systolic geometry. Proceedings of the International Congress of Mathematicians. Volume II, 745-768, Hindustan Book Agency, New Delhi, 2010.

\bibitem[HW]{HW} Hurewicz, W. and Wallman, H., {\it  Dimension Theory}, 
Princeton Mathematical Series, v. 4. Princeton University Press, Princeton, N. J., 1941. vii+165 pp. 

\bibitem[K]{K} Katz, M., Katz, M., The filling radius of two-point homogeneous spaces,
J. Differential Geom. 18 (1983) no. 3, 505-511.

\bibitem[MS]{MS} Michael, J.; Simon, L., Sobolev and mean-value inequalities
on generalized submanifolds of $\mathbb{R}^n$, Comm. Pure Appl. Math 26
(1973), 361-379.

\bibitem[W]{W} Wenger, S., A short proof of Gromov's filling inequality. Proc. Amer. Math. Soc. 136 (2008), no. 8, 2937-2941. 

\end{thebibliography}
\end{document}